\documentclass[final]{siamltex}

\def\d{\delta}

\usepackage{amssymb}
\usepackage{graphicx}
\usepackage{amsmath}
\usepackage{amssymb}
\usepackage[dvips]{color}
\usepackage{booktabs}
\usepackage{bbm}

\newtheorem{priteo}{Theorem}
\newtheorem{lema}{Lemma}

\newtheorem{defi}{Definition}

\title{Hopf bifurcation and heteroclinic cycles in a class of $\mathbb{D}_2-$equivariant systems\thanks{This
        work was supported by FCT grant $SFRH/ BD/ 64374/ 2009$.}}

\author{Adrian C. Murza\thanks{Centro de Matem\'atica da Universidade do Porto, Rua do Campo Alegre 687, 4169-007 Porto, Portugal}}

\begin{document}

\maketitle

\begin{abstract}
In this paper we analyze a generic dynamical system with $\mathbb{D}_2$ constructed via a Cayley graph. We study the Hopf bifurcation and find conditions for obtaining a unique branch of periodic solutions. Our main result comes from analyzing the system under weak coupling, where we identify the conditions for heteroclinic cycle between four equilibria in the two-dimensional fixed point subspace of some of the isotropy subgroups of $\mathbb{D}_2\times\mathbb{S}^1.$ We also analyze the stability of the heteroclinic cycle.
\end{abstract}



\pagestyle{myheadings}
\thispagestyle{plain}
\markboth{Adrian C. Murza}{Hopf bifurcation and heteroclinic cycles in a class of $\mathbb{D}_2-$equivariant systems}
\section{Introduction}
The global dynamics of networks of $n$ coupled oscillators with different types of coupling has been studied in \cite{Ashwin_Swift}. In their formalism it is shown that the symmetry group of the network can be considered a subgroup of $\mathbb{S}_n$, as long as the oscillators taken individually have no internal symmetries. These ideas have been investigated in the above cited paper, for general $\mathbb{D}_n$ and $\mathbb{Z}_n$ cases. However, it is a routine considering $n\geqslant3$ so the case $\mathbb{D}_2$ is sometimes not explicitly taken into account. Another reason for carrying out our study is the fact that among the dihedral groups, $\mathbb{D}_2$ is the only abelian group, which makes its action and therefore its analysis slightly different.

In this paper we are concerned with two properties of networks with $\mathbb{D}_2$ symmetry: Hopf bifurcation and low coupling case leading to heteroclinic cycles.
Firstly we use the methodology developed by Ashwin and Stork \cite{Stork} to construct a network of differential systems with $\mathbb{D}_2$ symmetry. Our case is a particular situation of a network of $n$ coupled oscillators with symmetry to be a subgroup of $\mathbb{S}_n,$ that analyzes different types of coupling between the oscillators, as shown in \cite{Stork}. While this approach has been successfully used for the subgroups $\mathbb{Z}_n$ and $\mathbb{D}_n,$ our approach is interesting not only because of the particularities already mentioned of $\mathbb{D}_2,$ but also because it offers the possibility of analyzing the weak coupling limit, where dynamics of the network is governed only by the phases of the oscillators.\\

A large variety of subgroups of $\mathbb{S}_n$ can be easily generated by the method described by Ashwin and Stork \cite{Stork} based on graph theory. The automorphism group of the colored graph can be defined to be its symmetry group. Therefore, it is possible, via Cayley graphs, to design oscillatory networks with the prescribed symmetry of a subgroup of $\mathbb{S}_n.$
While the Hopf bifurcation is in fact a simple adaptation of the theory developed by Golubitsky, Stewart and Schaeffer in \cite{GS85},\cite{GS86} and \cite{GS88}, the low coupling case is much more interesting since it allows the existence of heteroclinic cycles in systems with $\mathbb{D}_2$ symmetry. In this case it is possible to reduce the asymptotic dynamics to a flow on an four-dimensional torus $\mathbb{T}^4;$ by assuming a weak coupling we average the whole network and introduce an extra $\mathbb{S}^1$ symmetry.

We determine the two-dimensional invariant subspaces on this torus and show that a heteroclinic cycle consisting of four one-dimensional routes among the four zero-dimensional equilibria can appear. We then apply the general stability theory for heteroclinic cycles developed by Krupa and Melbourne in \cite{Krupa}, to analyze the stability of the heteroclinic cycles.

\section{The Cayley graph of the $\mathbb{D}_2$ group}\label{section Cayley}
Our aim is to construct an oscillatory system with the $\mathbb{D}_2$ symmetry. To achieve this goal we need first to represent the group by a Cayley diagram. A Cayley diagram is a graph (that is, a set of nodes and the arrows between them) to represent a group. Vertex or nodes of the graph are the group elements and the arrows show how the generators act on the elements of the group. At any vertex there are arrows pointing towards it, others pointing away. Proceeding as in \cite{ADSW} or \cite{Stork}, let $J\subset\mathbb{D}_2$ be the generating set of $\mathbb{D}_2.$ This implies that
\begin{itemize}
\item[(a)] $J$ generates $\mathbb{D}_2,$
\item[(b)] $J$ is finite and
\item[(c)] $J=J^{-1}.$
\end{itemize}
Since $\mathbb{D}_2$ is finite we can forget assumptions $(b)$ and $(c).$
Then the Cayley graph of $\mathbb{D}_2$, is a directed colored graph whose arrows indicate the action the group elements have on each other and the vertices are the group elements. As shown in \cite{ADSW} the generating set $J$ is a set of "colors" of the directed edges and the group elements $\sigma_i$ and $\sigma_j$ are connected through an edge from $\sigma_i$ to $\sigma_j$ of color $c\in J$ if and only if  $\sigma_i=c\sigma_j.$

\begin{figure}[ht]
\centering
\begin{center}
\includegraphics[width=7.77cm,height=7.72cm]{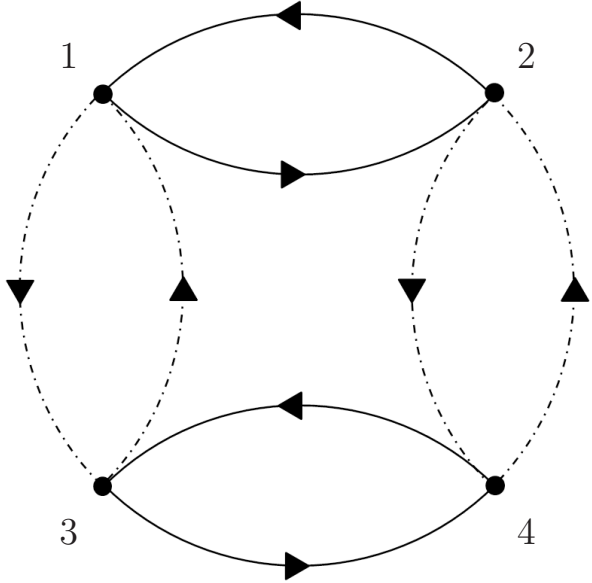}
\caption{A Cayley graph of the $\mathbb{D}_2$ group. Solid arrows represent left-multiplication with $\kappa,$ dot-and-dashed arrows left multiplication with $\zeta,$ the two generators of this group.}\label{second_figure}
\begin{picture}(0,0)\vspace{0.5cm}
\put(-90,235){\Large{$1$}}
\put(83,235){\Large{$2$}}
\put(-90,55){\Large{$3$}}
\put(83,55){\Large{$4$}}
\end{picture}\vspace{-0.3cm}
\end{center}
\end{figure}

The Cayley graphs for $\mathbb{D}_2$ is shown in Figure \eqref{second_figure}. Following the approach by Stork et al. \cite{Stork}, we may identify the vertices as cells with a certain dynamics and the edges as couplings between the cells. In this way we can construct an ODE system which has the symmetry of  $\mathbb{D}_2$. The action of the group $\mathbb{D}_2$ on the cells can be written as
\begin{equation}\label{elements}
\begin{array}{l}
\kappa=(1~2)(3~4)\in\mathbb{S}_4,\\
\zeta=(1~3)(2~4)\in\mathbb{S}_4,
\end{array}
\end{equation}
where the two generators and their commuting product $\kappa\zeta=\zeta\kappa$ act on $(x_1,x_2,x_3,x_4)\in\mathbb{R}^4$ as
\begin{equation}\label{commutators}
\begin{array}{l}
\kappa:(x_1,x_2,x_3,x_4)\rightarrow(x_2,x_1,x_4,x_3)\\
\zeta:(x_1,x_2,x_3,x_4)\rightarrow(x_3,x_4,x_1,x_2)\\
\kappa\zeta:(x_1,x_2,x_3,x_4)\rightarrow(x_4,x_3,x_2,x_1)
\end{array}
\end{equation}

If we assign coupling between cells related by the permutations in \eqref{elements}, we can build the following pairwise system in   with the $\mathbb{D}_2$ symmetry.

\begin{equation}\label{array 4 eq}
\begin{array}{l}
\dot{x}_1=f(x_1)+g(x_2,x_1)+h(x_3,x_1)\\
\dot{x}_2=f(x_2)+g(x_1,x_2)+h(x_4,x_2)\\
\dot{x}_3=f(x_3)+g(x_4,x_3)+h(x_1,x_3)\\
\dot{x}_4=f(x_4)+g(x_3,x_4)+h(x_2,x_4)\\
\end{array}
\end{equation}
where $f:\mathbb{R}\rightarrow\mathbb{R}$ and $g,~h:\mathbb{R}^2\rightarrow\mathbb{R}$. As shown by Ashwin and Stork \cite{Stork} we can think of $f,~g,~h$ as being generic functions that assure that the isotropy of this vector field under the action of $\mathbb{O}_4$ is generically  $\mathbb{D}_2$.

\section{Hopf bifurcation}\label{section Hopf bifurcation}

The group $\mathbb{D}_2$ has order $4$, the generators being $(\kappa,~\zeta).$ Since all irreducible representations of the $\mathbb{D}_2$ group are one dimensional, we restrict our study to the actions of the generators on $z\in\mathbb{C}\equiv\mathbb{R}^2.$ The orthogonal representation of $\mathbb{D}_2$
$$R_{\mathbb{D}_2}:\mathbb{D}_2\times W\rightarrow W$$
on the complex vector space $W$ is irreducible if and only if the only $\mathbb{D}_2-$invariant subspaces of $W$ are the trivial ones. In fact it can be shown that
$$\mathbb{D}_2\times\mathbb{S}^1/\mathrm{ker}R_{\mathbb{D}_2\times\mathbb{S}^1}\equiv\mathbb{Z}_2\times\mathbb{Z}_2\times\mathbb{S}^1/\mathrm{ker}R_{\mathbb{Z}_2\times\mathbb{Z}_2\times\mathbb{S}^1}.$$
The group element $(\kappa,~\zeta)$ acts on the point $z\in\mathbb{C}$ by
\begin{equation}\label{action on C}
\begin{array}{l}
\kappa\cdot z=\bar{z}\\
\zeta\cdot z=\pm z.
\end{array}
\end{equation}
The orbit of a point $(x,y)\in\mathbb{R}^2$ under the action of $\mathbb{D}_2$ is
$$\{(\kappa,\zeta)\cdot z|(\kappa,\zeta)\in\mathbb{D}_2\}.$$
Therefore the orbits are
\begin{itemize}
\item[(a)] The origin, $(0,0),$
\item[(b)] Points on the $x$-axis, $\pm x,0$ with $(x\neq0),$
\item[(c)] Points on the $y$-axis, $0,\pm y$ with $(y\neq0),$
\item[(d)] Points off the axes $\pm x,\pm y$ with $(x\neq0,~y\neq0).$
\end{itemize}

The isotropy subgroup of a point $(x,y)\in\mathbb{R}^2$ under the action of $\mathbb{D}_2$ is
$$\{(\kappa,~\zeta)\in\mathbb{D}_2|(\kappa,~\zeta)\cdot(x,y)=(x,y)\}.$$
$\mathbb{D}_2$ has four isotropy subgroups:
\begin{itemize}
\item[(a)] $\mathbb{D}_2$ corresponding to the origin
\item[(a)] $\mathbb{Z}_2=\{(1,\zeta)\}$ corresponding to $(x,0)$ with $x\neq0,$
\item[(a)] $\mathbb{Z}_2=\{(\kappa,1)\}$ corresponding to $(0,y)$ with $y\neq0,$
\item[(a)] $\mathbbm{1}$ corresponding to $(x,y)$ with $x\neq0,~y\neq0.$
\end{itemize}

In this section we basically recall the theory of Golubitsky and Stewart \cite{GS85} and \cite{GS88} on one parameter Hopf bifurcations with symmetry.

\subsection{General considerations on Hopf bifurcation with $\mathbb{D}_2\times\mathbb{S}^1-$symmetry}

The group $\mathbb{D}_2$ only has four $1-$dimensional real irreducible representations and so they are absolutely irreducible. One of them is the trivial representation and the other three are non-isomorphic each with a nontrivial $\mathbb{Z}_2$ action. Therefore, the only possible way to have a Hopf bifurcation is for the linearization at the equilibrium to be from a $\Gamma-$simple subspace of the form $V\oplus V$ where $V$ is an absolutely irreducible representation which has either a trivial or a $\mathbb{Z}_2$ action. Generically, as a parameter is varied, the critical eigenspace with purely imaginary eigenvalues is two-dimensional of the form $V\oplus V$ described above. The periodic solution bifurcating is unique and has the same symmetry has the $\mathbb{Z}_2$ action on V, or has no symmetry.

In fact, the $\mathbb{D}_2$ symmetric coupled system does not support any periodic solution. One can easily show that the isotypic decomposition at a trivial steady-state is made up of the trivial representation $(1,1,1,1)$ and the complement decomposes into the remaining three non-isomorphic irreducible representations. Therefore, the $V\oplus V$ subspace cannot exist in the tangent space at the trivial equilibrium. Each cell in the network needs to be two-dimensional to hope to have a Hopf bifurcation with $\mathbb{D}_2$ symmetry in such a $4-$cell system.

\begin{priteo}
For the $\mathbb{D}_2-$equivariant Hopf bifurcation in $\mathbb{C}$, there is a unique branch of periodic solutions consisting of rotating waves with $\mathbb{Z}_2$ spatial symmetry.
\end{priteo}
\begin{proof}
First, unicity of a branch of periodic solutions is based on the fact that $\mathbb{D}_2$ is abelian, and we recall the Abelian Hopf Theorem in \cite{AbelianHopf}. Therefore, the periodic solution predicted by the Equivariant Hopf Theorem \cite{GS88} with the symmetry of the unique isotropy subgroup with a two-dimensional fixed point subspace is the only one.
\end{proof}\\

\section{Normal form and generic linearized stability}
We begin with recalling results on the normal form for the $\mathbb{D}_2-$equivariant bifurcation from \cite{GS85} on the $\mathbb{Z}_2\times\mathbb{Z}_2$ group, which is isomorphic to $\mathbb{D}_2.$ This means that we can adapt the theory of $\mathbb{Z}_2\times\mathbb{Z}_2$ to our case, withe merely a re-interpretation of the branches. We say that the bifurcation problem $g$ commutes with $\mathbb{D}_2$ if
\begin{equation}\label{commut}
g((\kappa,\zeta)\cdot(x,y),\lambda)=(\kappa,\zeta)\cdot g(x,y,\lambda).
\end{equation}
\begin{lema}\label{lema normal form}
Let us consider the $g:\mathbb{C}\times\mathbb{R}\rightarrow\mathbb{R}$ bifurcation problem with $z=x+iy$ commuting with the action of $\mathbb{D}_2$. Then there exist smooth functions $p(u,v,\lambda),~q(u,v,\lambda)$ such that\\
\begin{equation}\label{eq lema normal form}
\begin{array}{l}
g(x,y,\lambda)=(p(x^2,y^2,\lambda)x,~q(x^2,y^2,\lambda)y),~~\mathrm{where}\\
\\
\hspace{2cm}p(0,0,0)=0,~~~q(0,0,0)=0.
\end{array}
\end{equation}
\end{lema}

\begin{proof}
The proof is almost identical to the one for $\mathbb{Z}_2\times\mathbb{Z}_2$ in \cite{GS85} but we prefer to give it adapted to the action of the group generators of $\mathbb{D}_2.$
We write $g$ as
\begin{equation}\label{eq lema g}
g(x,y,\lambda)=(a(x,y,\lambda),~b(x,y,\lambda)).
\end{equation}
Commutativity with equation \eqref{commut} implies
\begin{equation}\label{eq lema normal form2}
\begin{array}{l}
a(\kappa x,\zeta y,\lambda)=\kappa a(x,y,\lambda)\\
\\
b(\kappa x,\zeta y,\lambda)=\zeta b(x,y,\lambda)
\end{array}
\end{equation}
Now $\kappa$ transforms $z$ into $\bar{z}$, i.e. $(x,y)\rightarrow(x,-y)$ and if the action of $\zeta$ is $+1$ then we get that $a$ is odd in $x$ while $b$ is even in $x.$ When $\kappa$ acts as identity and $\zeta$ acts as $-1$ then $a$ is even in $y$ and $b$ is odd in $y.$ Now the rest of the proof is exactly the same as in \cite{GS85}.
\end{proof}\\
\\
In the following we will follow again the route and results in \cite{GS85} to obtain the Jacobian matrix of equation \eqref{eq lema normal form}
\begin{equation}\label{equation dg}
dg=
\begin{array}{l}
\begin{bmatrix}
p+2up_u&2p_vxy\\
2q_uxy&q+2vq_v
\end{bmatrix}
\end{array}.
\end{equation}
What is more, we can reformulate or adapt to our case Lemma $3.1$ (page $427$ in \cite{GS85}) and add the proof again corresponding to our case, which in the cited reference has been left as an exercise for the reader.
\begin{lema}\label{lema trivial}
Let $(x,y,\lambda)$ be a solution to $g=0.$
\begin{itemize}
\item[(a)] If $f(x,y,\lambda)$ is a trivial or a pure mode solution, then $dg$ in \eqref{equation dg} is diagonal and its eigenvalues which are real have the signs listed as follows:
    \begin{itemize}
    \item [(i)] Trivial solution: $\mathrm{sgn}~p(0,0,\lambda),~\mathrm{sgn}~q(0,0,\lambda);$
    \item [(ii)] $x-$mode solution: $\mathrm{sgn}~p_u(x,0,\lambda),~\mathrm{sgn}~q(x,0,\lambda);$
    \item [(iii)] $y-$mode solution: $\mathrm{sgn}~p(0,y,\lambda),~\mathrm{sgn}~q_v(0,y,\lambda);$
    \end{itemize}
\item[(b)] If $f(x,y,\lambda)$ is a mixed mode solution, then
\begin{itemize}
\item[(a)] $\mathrm{sgn~det}(dg)=\mathrm{sgn}(p_uq_v-p_vq_u)$ at $(x,y,\lambda);$
\item[(b)] $\mathrm{sgn~tr}(dg)=\mathrm{sgn}(up_u+vq_u)$ at $(x,y,\lambda).$
\end{itemize}
\end{itemize}
\end{lema}
\begin{proof}
By using the terminology in \cite{GS85}, we have
\begin{itemize}
\item[(a)] trivial solution when $x=y=0,$
\item[(b)] $x-$mode solution when $p(u,0,\lambda)=y=0,~x\neq0,$
\item[(c)] $y-$mode solution when $q(0,v,\lambda)=x=0,~y\neq0,$
\item[(d)] mixed-mode solution when $p(u,v,\lambda)=q(u,v,\lambda)=0,~x\neq0,~y\neq0.$
\end{itemize}
Accordingly, if $x=y=0$ equation \eqref{equation dg} reduces to
\begin{equation}\label{equation dg1}
dg=
\begin{array}{l}
\begin{bmatrix}
p&0\\
0&q
\end{bmatrix}
\end{array}
\end{equation}
in the first case. Moreover, since $y=0$ for the $x-$mode solution, then $v=0,~x\neq0$ and therefore the Jacobian matrix becomes
\begin{equation}\label{equation dg2}
dg=
\begin{array}{l}
\begin{bmatrix}
2up_u&0\\
0&q
\end{bmatrix}
\end{array}.
\end{equation}
Similarly, since $x=0$ for the $y-$mode solution, then $u=0,~y\neq0$ and therefore the Jacobian matrix becomes
\begin{equation}\label{equation dg4}
dg=
\begin{array}{l}
\begin{bmatrix}
p&0\\
0&2vq_v
\end{bmatrix}
\end{array}.
\end{equation}

Finally, when we deal with the mixed-mode solutions then the Jacobian matrix becomes
\begin{equation}\label{equation dg5}
dg=
\begin{array}{l}
\begin{bmatrix}
2up_u&2p_vxy\\
2q_uxy&2vq_v
\end{bmatrix}
\end{array},
\end{equation}

and the eigenvalues are
\begin{equation}\label{eigenvalues normal form}
\begin{array}{l}
\lambda_{1,2}=\displaystyle{\frac{2up_u+2vq_v\pm2\sqrt{u^2p_u^2+6uvp_uq_v+v^2q_v^2}}{2}}=\\
\\
\hspace{3cm}up_u+vq_v\pm\sqrt{\left(up_u+vq_v\right)^2+4uvp_uq_v}.
\end{array}
\end{equation}
Now since $u=x^2$ and $v=y^2$ we explicitly derive the conditions for stability/unstability of the linearized system, in Table \eqref{table sign jacobian}.
\begin{table}
\centering
\begin{center}
\caption{Mixed-mode solutions of the $\mathbb{D}_2-$equivariant normal of equation \eqref{eq lema normal form}.}\label{table sign jacobian}
\end{center}
\begin{tabular}{cccc}
\toprule
$(p_u,q_v)$ & $up_u+vq_v$& $\left(up_u+vq_v\right)^2+4uvp_uq_v$&Stability \\
\midrule
\centering
$(+,+)$ &+&+&unstable\\
$(-,-)$ &-&+&stable\\
$(+,-)$& +&+&unstable\\
$(+,-)$& +&-&unstable\\
$(+,-)$& -&+&stable\\
$(+,-)$& -&-&stable\\
$(-,+)$& +&+&unstable\\
$(-,+)$& +&-&unstable\\
$(-,+)$& -&+&stable\\
$(-,+)$& -&-&stable\\
\bottomrule
\end{tabular}
\end{table}
\end{proof}

\section{Weak Coupling}\label{section Weak Coupling}

We can think on ideal physical systems as interactions between identical or nearly identical subsystems. Moreover, as in \cite{Ashwin_Swift}, the whole system can be described as being a perturbation of an uncoupled system. This can be thought of as continuous path between the system and a system that is a product of several dynamical systems. In the following we will assume that our system of oscillators has a weak coupling limit. In fact, as it has been shown in Ashwin and Swift \cite{Ashwin_Swift}, even strongly coupled oscillator systems must have a weakly coupled limit. Moreover, we will assume that our system is formed by dissipative oscillators, so that the periodic orbit is attracting and unique in some neighborhood. In the weak coupling case, we focus on the dynamics of the relative phases of the oscillators. This situation can be better understood when in the no coupling case there is an attracting $N-$torus with one angle for each oscillator. An apparent problem can arise when considering phase differences; however we can choose coordinates so that the dynamics is a linear flow in the direction and there is no change in the phase differences. The theory developed in \cite{Ashwin_Swift} shows that the torus is normally hyperbolic, so with small coupling, the torus persists and there is a slow evolution of the phase differences. Moreover, it has been pointed out that the weak coupling limit offers different information about the dynamics rather than analyzing only the small amplitude periodic solutions near the Hopf bifurcation point.

Our system \eqref{array 4 eq} can be rewritten under weak coupling case as an ODE of the form:
\begin{equation}\label{generic weak coupling equation}
\dot{x}_i=f(x_i)+\epsilon g_i(x_1,\ldots,x_4)
\end{equation}
for $i=1,\ldots,4,~x_i\in \mathcal{X}$ and commuting with the permutation action of $\mathbb{D}_2$ on $\mathcal{X}^4,$ both $f$ and $g_i$ being of the class $\mathcal{C}^{\infty}.$ The constant $\epsilon$ represents the coupling strength and it is assumed to have low values. As in \cite{Ashwin_Swift}, \cite{ADSW} or \cite{Stork} we assume $\dot{x}$ has an hyperbolic stable limit cycle.\\

As shown in \cite{Ashwin_Swift} in the case of weak coupling, there is a natural reason why we should not just look at irreducible representations of $\mathbb{D}_2$. In our case there are $4$ stable hyperbolic limit cycles in the limit of $\epsilon=0,$ which means that the asymptotic dynamics of the system factors into the asymptotic dynamics of four limit cycles. This way as it we show later, we can for example embed the flow of a $2-$dimensional torus on a four-dimensional torus $\mathbb{T}^4.$
Moreover we assume hyperbolicity of the individual limit cycles for small enough values of the coupling parameter, and this justifies expressing the dynamics of the system as an ODE in terms of four phases, i.e. an ODE on $\mathbb{T}^4$  which is $\mathbb{D}_2-$equivariant.
\begin{table}
\centering
\begin{center}
\caption{Isotropy subgroups and fixed point subspaces for the $\mathbb{D}_2\times\mathbb{S}^1$ action on $\mathbb{T}^4$. The generators are $\kappa=\{(0~1)(2~3)\}$ and $\zeta=\{(0~2)(1~3)\}.$ The isotropy subgroups act by reflection about the zero, one and two-dimensional manifolds on $\mathbb{T}^4.$ For example, by notation $\mathbb{Z}^{\phi}_2(\kappa_{0,0})$ we mean the group $\mathbb{Z}_2$ acting through reflection about the circle $x_0=x_1=0,x_2=x_3=\phi,$ while $\mathbb{Z}^{\phi_i}_2(\kappa_{0,0})$ whith $i=\{1,2\}$ acts by reflection about the disk $x_0=x_1=0,x_2=\phi_1,x_3=\phi_2$.$~(\phi,~\phi_1,~\phi_2\in\left[0,\pi\right])$.}\label{table grande}
\end{center}
\begin{tabular}{ccccc}
\toprule
$\Sigma$ & $\mathrm{Fix}(\Sigma)$ & Generators & $\mathrm{dim~Fix}(\Sigma)$\\
\midrule
\centering
$\mathbb{D}_2(0)$ & $(0,0,0,0)$ & $\{\kappa,~\zeta\}_{(0)}$ &0\\
$\mathbb{D}_2(\kappa,\mathrm{Id})$ & $(0,0,\pi,\pi)$ & $\{\kappa,~\mathrm{Id}\}_{(0,\pi)}$ &0\\
$\mathbb{D}_2(\kappa\zeta,\kappa)$ & $(0,\pi,0,\pi)$ & $\{\kappa\zeta,~\kappa\}_{(0,\pi)}$ &0\\
$\mathbb{D}_2(\pi)$& $(\pi,0,0,\pi)$ & $\{\kappa,~\zeta\}_{(\pi)}$ &0\\
$\mathbb{Z}^{\phi}_2(\kappa_{(0,0)})$ & $(0,0,\phi,\phi)$ & $\{(0~1)\}_{(0,0)}$ &1\\
$\mathbb{Z}^{\phi}_2(\kappa_{(\pi,\pi)})$ & $(\pi,\pi,\phi,\phi)$ & $\{(0~1)\}_{(\pi,\pi)}$ &1\\
$\mathbb{Z}^{\phi}_2(\zeta_{(\pi,\pi)})$ & $(\phi,\phi,\pi,\pi)$ & $\{(2~3)\}_{(\pi,\pi)}$ &1\\
$\mathbb{Z}^{\phi}_2(\zeta_{(0,0)})$ & $(\phi,\phi,0,0)$ & $\{(2~3)\}_{(0,0)}$ &1\\
$\mathbb{Z}^{\phi_i}_2(\kappa_{(0,0)})$ & $(0,0,\phi_1,\phi_2)$ & $\{(0~1)\}_{(0,0)}$ &2\\
$\mathbb{Z}^{\phi_i}_2(\kappa_{(\pi,\pi)})$ & $(\pi,\pi,\phi_1,\phi_2)$ & $\{(0~1)\}_{(\pi,\pi)}$ &2\\
$\mathbb{Z}^{\phi_i}_2(\zeta_{(\pi,\pi)})$ & $(\phi_1,\phi_2,\pi,\pi)$ & $\{(2~3)\}_{(\pi,\pi)}$ &2\\
$\mathbb{Z}^{\phi_i}_2(\zeta_{(0,0)})$ & $(\phi_1,\phi_2,0,0)$ & $\{(2~3)\}_{(0,0)}$ &2\\
\bottomrule
\end{tabular}
\end{table}

In addition, Ashwin and Swift showed in \cite{Ashwin_Swift} that for small enough values of the coupling parameter it is possible to average the equations and introduce an approximate decoupling between the fast variation of the phases and the slow variation of the phase differences. This can be seen as introducing and phase shift symmetry which acts on $\mathbb{T}^4$ by translation along the diagonal;
$$R_{\theta}(\phi_1,\ldots,\phi_4):=(\phi_1+\theta,\ldots,\phi_4+\theta),$$
for $\theta\in\mathbb{S}^1.$

Now have an ODE on that is equivariant under the action of $\mathbb{D}_2\times\mathbb{S}^1;$  now we have to classify the isotropy types of points under this action.

\begin{priteo}
The isotropy subgroups for the action of $\mathbb{D}_2\times\mathbb{S}^1$ on $\mathbb{T}^4$ together with their generators and dimension of their fixed-point subspaces are those listed in Table \eqref{table grande}.
\end{priteo}
\begin{proof}
We will explicitly calculate two examples, for the zero and one-dimensional fixed-point subspaces, respectively, the other cases being treated similarly.
$(a)$ Let's take the action of $\mathbb{D}_2(\kappa,\mathrm{Id})$ on $\mathbb{T}^4.$ We have
\begin{equation}\label{teorema tabla example}
\begin{array}{l}
\begin{bmatrix}
\cos\phi_1&-\sin\phi_1&0&0\\
\sin\phi_1&\cos\phi_1&0&0\\
0&0&\cos\phi_2&-\sin\phi_2\\
0&0&\sin\phi_2&\cos\phi_2\\
\end{bmatrix}
\begin{bmatrix}
a\\b\\c\\d
\end{bmatrix}=
\begin{bmatrix}
a\cos\phi_1-b\sin\phi_1\\
b\cos\phi_1+a\sin\phi_1\\
c\cos\phi_2-d\sin\phi_2\\
d\cos\phi_2+c\sin\phi_2\\
\end{bmatrix}=
\begin{bmatrix}
a\cos\phi\\
b\cos\phi\\
c\cos\phi\\
d\cos\phi\\
\end{bmatrix}=
\begin{bmatrix}
\pm a\\
\pm b\\
\pm c\\
\pm d\\
\end{bmatrix}.
\end{array}
\end{equation}
because $\phi_1,\phi_2=\{0,\pi\}.$ Therefore the only possible values for any arbitrary point on $\mathbb{T}^4$ to be fixed by the group are $0$ and $\pi.$ The four choices for the first four lines are deduced from the action of the elements of $\mathbb{D}_2.$\\

$(b)$ Let's take for example to action of $\mathbb{Z}^{\phi}_2(\zeta_{(\pi,0)})$ on one dimensional manifolds on $\mathbb{T}^4.$
\begin{equation}\label{teorema tabla example2}
\begin{array}{l}
\begin{bmatrix}
1&0&0&0\\
0&1&0&0\\
0&0&\cos\phi&-\sin\phi\\
0&0&\sin\phi&\cos\phi\\
\end{bmatrix}
\begin{bmatrix}
\phi\\
\phi\\
\pi\\
\pi\\
\end{bmatrix}=
\begin{bmatrix}
\phi\\
\phi\\
\pi\cos\phi\\
\pi\cos\phi\\
\end{bmatrix}=
\begin{bmatrix}
\phi\\
\phi\\
\pm\pi\\
\pm\pi
\end{bmatrix}=
\begin{bmatrix}
\phi\\
\phi\\
\pi\\
\pi\\
\end{bmatrix},
\end{array}
\end{equation}
because the only options for $\phi$ are $0$ or $\pi$ radians.

\end{proof}

\subsection{Analysis of a family of vector fields in $\mathrm{Fix}(\mathbb{Z}_2)$}\label{onetorus_theor}
We can define coordinates in $\mathrm{Fix}(\mathbb{Z}_2) $ by taking a basis
\begin{equation}\label{basis}
\begin{array}{l}
e_1=\frac{1}{2}(-1,1,-1,-1)\\
e_2=\frac{1}{2}(-1,-1,1,-1)
\end{array}
\end{equation}
and consider the space spanned by $\{e_1,e_2\}$ parameterized by $\{\phi_1,\phi_2\}:$
\begin{equation}\label{coordinates}
\sum_{n=1}^2\phi_ne_n
\end{equation}
By using these coordinates, we construct the following family of two-dimensional differential systems which satisfies the symmetry of $\mathrm{Fix}(\mathbb{Z}_2)$.
\begin{equation}\label{systema ejemplo}
\left\{
\begin{array}{l}
\dot{\phi_1}=a\sin{\phi_1}\cos{\phi_2}+\epsilon\sin{2\phi_1}\cos{2\phi_2}\\
\\
\dot{\phi_2}=-b\sin{\phi_2}\cos{\phi_1}+\epsilon\sin{2\phi_2}\cos{2\phi_1}+q(1+\cos\phi_1)\sin2\phi_2,\\
\end{array}
\right.
\end{equation}
where $a,~b>0.$
We argue that the family of vector fields \eqref{systema ejemplo} exhibits structurally stable, attracting heteroclinic cycles, which are structurally stable or completely unstable, depending on parameters $a,~b,~\epsilon$ and $q,$ as we shall prove in Theorem \ref{teorema estabilidad heteroclinas}.
In the following we will show that the planes $\phi_1=0~(\mathrm{mod}~\pi),~\phi_2=0~(\mathrm{mod}~\pi)$ are invariant under the flow of \eqref{systema ejemplo}.

Let $\mathcal{X}$ be the vector field of system \eqref{systema ejemplo}.\\
\begin{defi}
We call a trigonometric invariant algebraic curve $h(\phi_1,\phi_2)=0,$ if it is invariant by the flow of \eqref{systema ejemplo}, i.e. there exists a function $K(\phi_1,\phi_2)$ such that
\begin{equation}\label{campo}
\mathcal{X}h=\frac{\partial h}{\partial\phi_1}\dot{\phi_1}+\frac{\partial h}{\partial\phi_2}\dot{\phi_2}=Kh.
\end{equation}
\end{defi}

\begin{lema}\label{lema invariant2}
Functions $\sin\phi_1$ and $\sin\phi_2$ are trigonometric invariant algebraic curves for system \eqref{systema ejemplo}.
\end{lema}
\begin{proof}
We can write the system \eqref{systema ejemplo} in the form
\begin{equation}\label{systema ejemplo1}
\left\{
\begin{array}{l}
\dot{\phi_1}=\sin{\phi_1}\left(a\cos{\phi_2}+2\epsilon\cos{\phi_1}\cos{2\phi_2}\right)\\
\\
\dot{\phi_2}=\sin{\phi_2}\left(-b\cos{\phi_1}+2\epsilon\cos{\phi_2}\cos{2\phi_1}+2q(1+\cos\phi_1)\cos\phi_2\right)\\
\end{array}
\right.
\end{equation}

Now if we choose $h_1=\sin\phi_1,$ then $Xh_1=\cos{\phi_1}\sin{\phi_1}\left(a\cos{\phi_2}+2\epsilon\cos{\phi_1}\cos{2\phi_2}\right)$
so $K_1=\cos{\phi_1}\left(a\cos{\phi_2}+2\epsilon\cos{\phi_1}\cos{2\phi_2}\right).$ The second case follows similarly.
\end{proof}\\

Since the planes $\phi_i=0(\mathrm{mod}~\pi),~i=1,2$ are invariant under the flow of \eqref{systema ejemplo}, it is clear that $(0,0),~(\pi,0),~(\pi,\pi),~(0,\pi)$ are equilibria for \eqref{systema ejemplo}. To check the possibility of heteroclinic cycles in system \eqref{systema ejemplo}, we linearize first about the equilibria (i.e. the zero-dimensional fixed points) and then about the one-dimensional manifolds connecting these equilibria. We can assume without loss of genericity that $\mathrm{Fix}\left(\mathbb{Z}_2\right)$ is attracting for the dynamics and therefore the stabilities. The idea is that the analysis of the dynamics within the fixed point space $\mathrm{Fix}\left(\mathbb{Z}_2\right)$ is crucial in determining the stabilities of the full system. In particular we will prove that eigenvalues of the linearization in each cases are of opposite signs, allowing the existence of such a heteroclinic network between the equilibria.

In the following we will show that there exists the possibility of a heteroclinic cycle in any of the two-dimensional fixed-point spaces in Table \eqref{table grande} and the connections between these zero dimensional fixed point spaces are possible within specific routes within the one-dimensional fixed points of \eqref{systema ejemplo}. For the proof we will need
\begin{lema}\label{lema saddles}
Assume $|\epsilon|<\mathrm{min}\{\frac{a}{2},\frac{b}{2}\}$ and $|\epsilon+2q|<\frac{b}{2}.$ Then the four fixed points in Table \eqref{table fixed points} are saddles.
\end{lema}
\begin{proof}
By inspection it is clear that the eigenvalues of these fixed points are of opposite signs in each case.

\begin{table}
\centering
\begin{center}
\caption{Eigenvalues of the flow of equation \eqref{systema ejemplo}, at the four non-conjugate zero-dimensional fixed points.}\label{table fixed points}
\end{center}
\begin{tabular}{cccc}
\toprule
$\mathrm{Fix}(\Sigma)$ & $(\phi_1,\phi_2)$ & $\lambda_1$ & $\lambda_2$\\
\midrule
\centering
$\mathbb{D}_2(0)$& $(0,0)$ & $a+2\epsilon$ &$-b+2(\epsilon+2q)$\\
$\mathbb{D}_2(\kappa,\mathrm{Id})$ & $(0,\pi)$ & $-a+2\epsilon$ &$b+2(\epsilon+2q)$\\
$\mathbb{D}_2(\kappa\zeta,\kappa)$ & $(\pi,0)$ & $-a+2\epsilon$ &$b+2\epsilon$\\
$\mathbb{D}_2(\pi)$ & $(\pi,\pi)$ & $a+2\epsilon$ &$-b+2\epsilon$\\
\bottomrule
\end{tabular}
\end{table}

\end{proof}
\begin{priteo}
There exists the possibility of a heteroclinic cycle in any of the two-dimensional fixed point subspaces in Table \eqref{table grande} in the following way:
\begin{equation}\label{flechas}
\begin{array}{l}
\cdots\xrightarrow{\mathbb{Z}^{\phi}_2(\zeta_{(0,0)})}\mathbb{D}_2(0)\xrightarrow{\mathbb{Z}^{\phi}_2(\kappa_{(0,0)})}
\mathbb{D}_2(\kappa,\mathrm{Id})\xrightarrow{\mathbb{Z}^{\phi}_2(\zeta_{(\pi,\pi)})}\\
\\
\hspace{3cm}\mathbb{D}_2(\pi)
\xrightarrow{\mathbb{Z}^{\phi}_2(\kappa_{(\pi,\pi)})}\mathbb{D}_2(\kappa\zeta,\kappa)\xrightarrow{\mathbb{Z}^{\phi}_2(\zeta_{(0,0)})}\cdots
\end{array}
\end{equation}
where the connection between the four equilibria in the plane $\phi_1,\phi_2$ is carried out along the indicated one-dimensional manifolds.
\end{priteo}

\begin{proof}
From Lemma \eqref{lema saddles} we know that the four fixed points are saddles, so there is possible a heteroclinic connection between them. We linearize the system \eqref{systema ejemplo} at every point in the one-dimensional manifolds of the fixed point spaces in Table \eqref{table grande}.
The Jacobian matrix of the system evaluated at these points has eigenvalues shown in Table \eqref{table fixed points repetition}.
This way we obtain the four paths shown in Table \eqref{table fixed points}. Moreover, using the conditions for $\epsilon$ and $q$ in Lemma \eqref{lema saddles} these eigenvalues are clearly of opposite signs in each case. A schematic view of the heteroclinic cycle is offered in Figure \eqref{flechas repetition}.
\begin{table}
\centering
\begin{center}
\caption{Eigenvalues of the flow of equation \eqref{systema ejemplo}, at the four non-conjugate paths on the one-dimensional fixed-point subspaces on $\mathbb{T}^4$.}\label{table fixed points repetition}
\end{center}
\begin{tabular}{cccc}
\toprule
$\mathrm{Fix}(\Sigma)$ & $(\phi_1,\phi_2)$ & $\lambda_1$ & $\lambda_2$\\
\midrule
\centering
$\mathbb{Z}^{\phi_1}_2(\kappa_{(0,0)})$ &$(\phi_1,0)$& $a\cos\phi_1+2\epsilon\cos2\phi_1$ &$-b\cos\phi_1+2q(1+\cos\phi_1)+2\epsilon\cos2\phi_1$\\
$\mathbb{Z}^{\phi}_2(\zeta_{(\pi,\pi)})$& $(\pi,\phi_2)$& $-a\cos\phi_2+2\epsilon\cos2\phi_2$ &$b\cos\phi_2+2\epsilon\cos2\phi_2$\\
$\mathbb{Z}^{\phi}_2(\kappa_{(\pi,\pi)})$ & $(\phi_1,\pi)$ &$-a\cos\phi_1+2\epsilon\cos2\phi_1$&$b\cos\phi_1+2q(1+\cos\phi_1)+2\epsilon\cos2\phi_1$\\
$\mathbb{Z}^{\phi}_2(\zeta_{(0,0)})$& $(0,\phi_2)$ & $a\cos\phi_2+2\epsilon\cos2\phi_2$ &$-b\cos\phi_2+2\epsilon\cos2\phi_2+4q\cos2\phi_2$\\
\bottomrule
\end{tabular}
\end{table}

\end{proof}

\begin{figure}[h]
\centering
\begin{center}
\includegraphics[width=5.83cm,height=7.0cm]{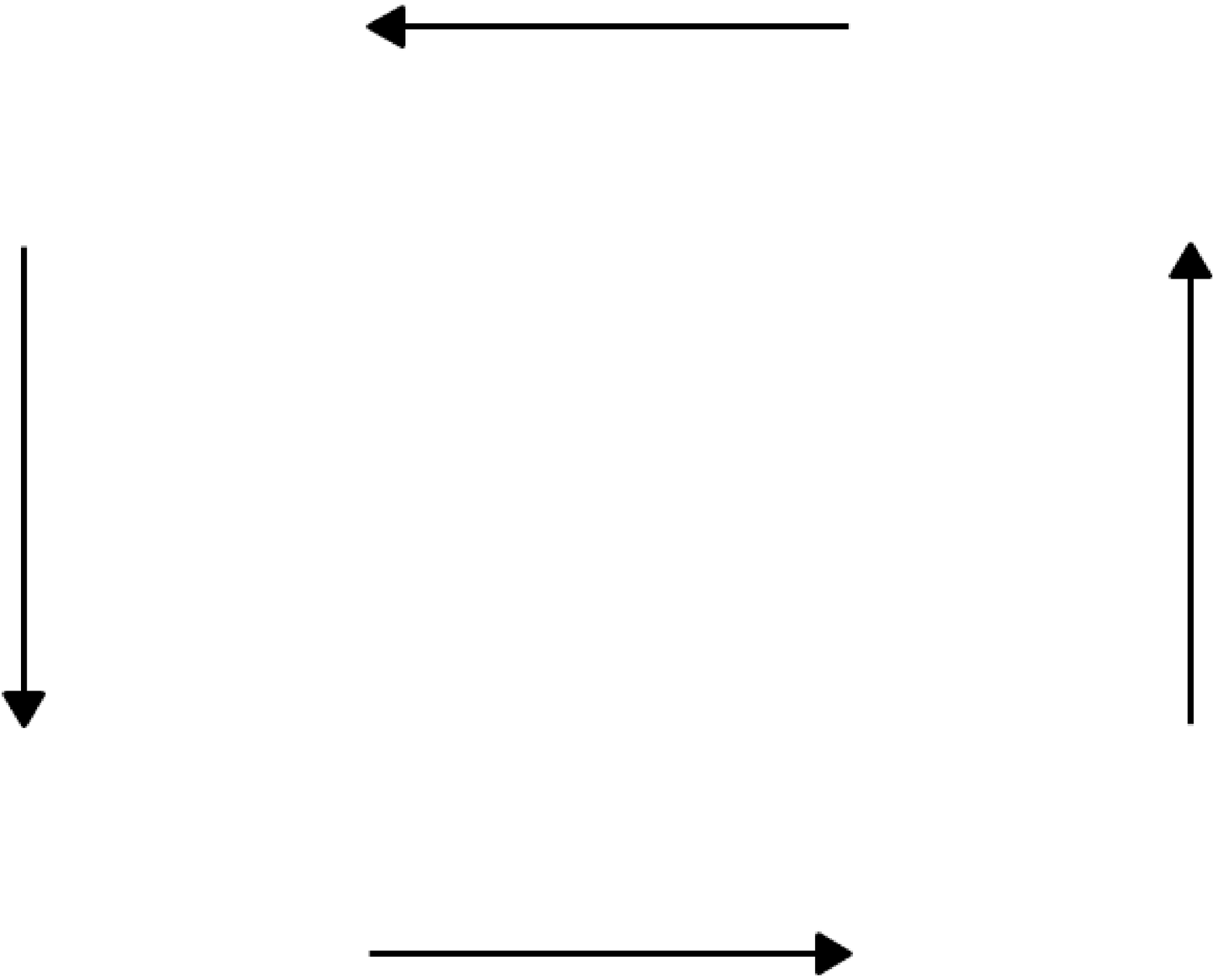}
\caption{Schematic representation of a structurally stable heteroclinic cycle for the weak coupled oscillators within the two-dimensional space of the $\mathrm{Fix}(\mathbb{Z}_2)$ group.}\label{flechas repetition}
\begin{picture}(0,0)
\put(-90,81){\Large{$\mathbb{D}_2(0)$}}
\put(50,81){\Large{$\mathbb{D}_2(\kappa,\mathrm{Id})$}}
\put(-30,60){\Large{$\mathbb{Z}^{\phi}_2(\kappa_{(0,0)})$}}
\put(87,169){\mbox{\rotatebox{-90}{{\Large{$\mathbb{Z}^{\phi}_2(\zeta_{(\pi,\pi)})$}}}}}
\put(58,197){\Large{$\mathbb{D}_2(\pi)$}}
\put(-25,213){\Large{$\mathbb{Z}^{\phi}_2(\kappa_{(\pi,\pi)})$}}
\put(-110,197){\Large{$\mathbb{D}_2(\kappa\zeta,\kappa)$}}
\put(-105,117){\mbox{\rotatebox{90}{{\Large{$\mathbb{Z}^{\phi}_2(\zeta_{(0,0)})$}}}}}
\end{picture}
\end{center}
\end{figure}

Under the conditions in Lemma \eqref{lema saddles}, the stability of the heteroclinic cycle is guaranteed by the fact that it is restricted to the dynamics of system \eqref{systema ejemplo} within the fixed point spaces. We use the criteria of Krupa and Melbourne \cite{Krupa} to study the stability of the heteroclinic cycle.
We now recall these stability criteria, which are based on four hypotheses.\\
The first hypothesis guarantees that the heteroclinic cycle is robust.
\begin{itemize}
\item ($S_1$) There is an isotropy subgroup $\Sigma_j$ with the fixed-point subspace $P_j=\mathrm{Fix}(\Sigma_j)$ such that $W^u(\xi_j)\cap P_j\subset W^s(\xi_{j+1})$ and $\xi_{j+1}$ is a sink in $P_j.$\\
Corresponding to each isotropy subgroup $\Sigma_j$ in ($S_1$) is the isotypic decomposition $\mathbb{R}^n=W_0\oplus\ldots\oplus W_q$ of $\mathbb{R}^n$ into isotypic components. Let $W_0=P_j$ and $N(\Sigma_j)$ denote the normalizer of $\Sigma_j$ in $\Gamma.$\\
\item ($S_2$) The eigenspaces corresponding to $c_j,~t_j,~e_{j+1}$ and $t_{j+1}$ lie in the same $\Sigma_j-$isotypic component;\\
\item ($S_3$) $\mathrm{dim}W^u(\xi_j)\cap P_j=\mathrm{dim}(N(\Sigma_j)/\Sigma_j)+1;$\\
\item ($S_4$) All transverse eigenvalues of $\xi_j$ with positive real part lie in the same $\Sigma_j-$isotypic component.
\end{itemize}

Set $\rho_j=\mathrm{min}(c_j/e_j,1-t_j/e_j)$ and define $\rho=\rho_1\cdot\cdot\cdot\rho_m.$

\begin{priteo}
[Krupa, Melbourne] Let $\Gamma$ be a finite group acting on $\mathbb{R}^n$ and $f:\mathbb{R}^n\rightarrow\mathbb{R}^n$ be a $\Gamma-$equivariant vector field. Suppose that $X$ is a heteroclinic cycle for $f$ satisfying hypotheses $(S_1)-(S_4).$ Then generically the stability of $X$ is described by precisely one of the following possibilities.
\begin{itemize}
\item [(a)] asymptotically stable ($\rho>1$ and $t_j<0$ for each $j$),
\item [(b)] unstable but essentially asymptotically stable ($\rho>1$ and $t_j<e_j$ for each $j$ and $t_j>0$ for some $j$),
\item [(c)] almost completely unstable ($\rho<1$ or $t_j>e_j$ for some $j$),
\item [(d)] completely unstable if $\rho<1.$
\end{itemize}
\end{priteo}

Applying these criteria to our case, we have

\begin{priteo}\label{teorema estabilidad heteroclinas}
Heteroclinic cycle \eqref{flechas} is:
\begin{itemize}
\item [(a)] asymptotically stable if
\begin{equation}\label{eq1 teorema estabilidad}
\begin{array}{l}
-\mathrm{min}\{\frac{a}{2},\frac{b}{2}\}<\epsilon<0~~\mathrm{and}~~-\frac{b}{2}<\epsilon+2q<0,~~\mathrm{or}\\
\\
-\mathrm{min}\{\frac{a}{2},\frac{b}{2}\}<\epsilon<0~~\mathrm{and}~~0<\epsilon+2q<\frac{b}{2}~~\mathrm{and}\\
\\
\hspace{5cm}\displaystyle{\frac{b-2(\epsilon+2q)}{b+2(\epsilon+2q)}>\frac{(a+2\epsilon )^2(b+2\epsilon)}{(a-2\epsilon )^2(b-2\epsilon)}},~~\mathrm{or}\\
\\
0<\epsilon<\mathrm{min}\{\frac{a}{2},\frac{b}{2}\}~~\mathrm{and}~~-\frac{b}{2}<\epsilon+2q<0~~\mathrm{and}\\
\\
\hspace{5cm}\displaystyle{\frac{b-2(\epsilon+2q)}{b+2(\epsilon+2q)}<\frac{(a+2\epsilon )^2(b+2\epsilon)}{(a-2\epsilon )^2(b-2\epsilon)}};
\end{array}
\end{equation}
\item [(b)] completely unstable if
\begin{equation}\label{eq1 teorema estabilidad repetition}
\begin{array}{l}
0<\epsilon<\mathrm{min}\{\frac{a}{2},\frac{b}{2}\}~~\mathrm{and}~~0<\epsilon+2q<\frac{b}{2},~~\mathrm{or}\\
\\
-\mathrm{min}\{\frac{a}{2},\frac{b}{2}\}<\epsilon<0~~\mathrm{and}~~0<\epsilon+2q<\frac{b}{2}~~\mathrm{and}\\
\\
\hspace{5cm}\displaystyle{\frac{b-2(\epsilon+2q)}{b+2(\epsilon+2q)}<\frac{(a+2\epsilon )^2(b+2\epsilon)}{(a-2\epsilon )^2(b-2\epsilon)}},~~\mathrm{or}\\
\\
0<\epsilon<\mathrm{min}\{\frac{a}{2},\frac{b}{2}\}~~\mathrm{and}~~-\frac{b}{2}<\epsilon+2q<0~~\mathrm{and}\\
\\
\hspace{5cm}\displaystyle{\frac{b-2(\epsilon+2q)}{b+2(\epsilon+2q)}>\frac{(a+2\epsilon )^2(b+2\epsilon)}{(a-2\epsilon )^2(b-2\epsilon)}}.
\end{array}
\end{equation}
\end{itemize}
\end{priteo}
\begin{proof}

The stability is expressed by
\begin{equation}\label{stability krupa1}
\rho=\prod_{j=1}^4\rho_j
\end{equation}
where
\begin{equation}\label{stability krupa2}
\rho_j=\mathrm{min}\{c_j/e_j,1-t_j/e_j\}.
\end{equation}
In equation \eqref{stability krupa2}, $e_i$ is the expanding eigenvector at the $i$th point of the cycle, $-c_i$ is the contracting eigenvector and $t_i$ is the tangential eigenvector of the linearization.
For the heteroclinic cycle we have
\begin{equation}\label{rho values1}
\begin{array}{l}
\rho_1=\displaystyle{\frac{b-2(\epsilon+2q)}{a+2\epsilon}}~~
\rho_2=\displaystyle{\frac{a-2\epsilon}{b+2(\epsilon+2q)}}~~
\rho_3=\displaystyle{\frac{a-2\epsilon}{b+2\epsilon}}~~
\rho_4=\displaystyle{\frac{b-2\epsilon}{a+2\epsilon}},
\end{array}
\end{equation}
so from equation \eqref{stability krupa1} we obtain
\begin{equation}\label{rho values2}
\begin{array}{l}
\rho=\displaystyle{\frac{\left[b-2(\epsilon+2q)\right](a-2\epsilon )^2(b-2\epsilon)}{\left[b+2(\epsilon+2q)\right](a+2\epsilon )^2(b+2\epsilon)}}.
\end{array}
\end{equation}
Then the proof follows by applying Theorem $2.4$ in \cite{Krupa}.
\end{proof}\\
\paragraph{\bf Acknowledgements}
The author would like to thank the referee for indications which improved the presentation of this paper. He also acknowledges financial support from FCT grant $SFRH/ BD/ 64374/ 2009.$

\end{document}